\documentclass[12pt]{article}

\usepackage{amssymb,amsfonts,amsthm,amsmath}
\usepackage{tikz,color,pgf,graphicx,url}
\usetikzlibrary{positioning}
\usepackage{enumerate}     
\usepackage[T1]{fontenc}
\usepackage{xcolor} 
\usepackage{txfonts} 
\usepackage{hyperref}
\hypersetup{
	colorlinks   = true, 
	urlcolor     = black, 
	linkcolor    = black, 
	citecolor   = black 
}

\newtheorem{theorem}{Theorem}[section]
\newtheorem{corollary}[theorem]{Corollary}

\newtheorem{proposition}[theorem]{Proposition}
\newtheorem{conjecture}[theorem]{Conjecture}
\newtheorem{obs}[theorem]{Observation}
\newtheorem{problem}{Problem}

\newtheorem{remark}[theorem]{Remark}

\def\cp{\,\square\,}

\DeclareMathOperator{\diam}{diam}
\DeclareMathOperator{\rad}{rad}
\DeclareMathOperator{\ecc}{ecc}

\tikzstyle{vertex}=[circle, draw, inner sep=0pt, minimum size=6pt]

\begin{document}

\title{On $d$-distance $p$-packing domination number in strong products}

\author{
Csilla Bujt\'as $^{a,b,}$\thanks{Email: \texttt{csilla.bujtas@fmf.uni-lj.si}}
\and
Vesna Ir\v si\v c Chenoweth $^{a,b,}$\thanks{Email: \texttt{vesna.irsic@fmf.uni-lj.si}}
\and
Sandi Klav\v zar $^{a,b,c,}$\thanks{Email: \texttt{sandi.klavzar@fmf.uni-lj.si}}
\and
Gang Zhang $^{b,d,}$\thanks{Email: \texttt{gzhang@stu.xmu.edu.cn}}
}

\maketitle

\begin{center}
$^a$ Faculty of Mathematics and Physics, University of Ljubljana, Slovenia\\
\medskip

$^b$ Institute of Mathematics, Physics and Mechanics, Ljubljana, Slovenia\\
\medskip

$^c$ Faculty of Natural Sciences and Mathematics, University of Maribor, Slovenia\\
\medskip

$^d$ School of Mathematical Sciences, Xiamen University, China\\
\medskip
\end{center}

\begin{abstract}
The $d$-distance $p$-packing domination number $\gamma_d^p(G)$ of a graph $G$ is the cardinality of a smallest set of vertices of $G$ which is both a $d$-distance dominating set and a $p$-packing. If no such set exists, then we set $\gamma_d^p(G) = \infty$. For an arbitrary strong product $G\boxtimes H$ it is proved that $\gamma_d^p(G\boxtimes H) \le \gamma_d^p(G) \gamma_d^p(H)$. By proving that $\gamma_d^p(P_m \boxtimes P_n) = \left \lceil \frac{m}{2d+1} \right \rceil \left \lceil \frac{n}{2d+1} \right \rceil$, and that if $\gamma_d^p(C_n) < \infty$, then $\gamma_d^p(P_m \boxtimes C_n) = \left \lceil \frac{m}{2d+1} \right \rceil \left \lceil \frac{n}{2d+1} \right \rceil$, the sharpness of the upper bound is demonstrated. On the other hand, infinite families of strong toruses are presented for which the strict inequality holds. For instance, we present strong toruses with difference $2$ and demonstrate that the difference can be arbitrarily large if only one factor is a cycle. It is also conjectured that if $\gamma_d^p(G) = \infty$, then $\gamma_d^p(G\boxtimes H) = \infty$ for every graph $H$. Several results are proved which support the conjecture, in particular, if $\gamma_d^p(C_m)= \infty$, then $\gamma_d^p(C_m \boxtimes C_n)=\infty$.  
\end{abstract}


\medskip\noindent
{\bf AMS Subj.\ Class.\ (2020):}  05C69, 05C76

\section{Introduction}
\label{sec:intro}

Let $G = (V(G), E(G))$ be a graph, $d$ and $p$ nonnegative integers, and $S\subseteq V(G)$. The set $S$ is said to be a {\em $d$-distance dominating set} of $G$ if for every vertex $u\in V(G)\setminus S$ there exists a vertex $x\in S$ such that $d_G(u,x) \le d$, and $S$ is said to be a {\em $p$-packing} of $G$ if $d_G(x,y) \ge p+1$ for every different vertices $x,y\in S$. As usual, here and later $d_G(x,y)$ denotes the shortest-path distance between $x$ and $y$. The cardinality of a smallest set $S$ which is both a $d$-distance dominating set and a $p$-packing is the \emph{$d$-distance $p$-packing domination number} $\gamma_d^p(G)$ of $G$. If $d <p$, it might happen that no such set exists. In this case, we set $\gamma_d^p(G) = \infty$. A smallest $d$-distance $p$-packing dominating set will be briefly called a {\em $\gamma_d^p$-set}. We may also briefly call a $d$-distance $p$-packing dominating set a {\em $d$-distance $p$-packing set}.

The $d$-distance $p$-packing domination number was introduced by Beineke and Henning~\cite{Beineke-1994}. At that time, in 1994, they suggested to denote it by $i_{p,d}(G)$ and call it the $(p,d)$-domination number. To place this general concept within the trends of the contemporary graph theory, our present terminology and notation was suggested in~\cite{BIKZ-2025a}. This approach broadly generalizes various concepts. One of them is the concept of perfect $d$-codes which we will discuss in Section~\ref{sec:conjecture}. Moreover,  $\gamma_1^0 = \gamma$ is the usual domination number, $\gamma_d^0 = \gamma_d$ is the $d$-distance domination number~\cite{henning-2020}, $\gamma_1^1$ is the independent domination number~\cite{Haynes-2023}, $\gamma_d^1$ is the $d$-distance independent domination number~\cite{Gimbel-1996, henning-2020}, $\gamma_2^2$ is the lower packing number~\cite{Henning-1998-1}, and $\gamma_d^d(G)$ is the $d$-independent $d$-domination number~\cite{Henning1991}. 

Let's quickly summarize what has been done recently in~\cite{BIKZ-2025a, BIKZ-2025b}. In the first paper it was shown that for every fixed $d$ and $p$, where $2 \le d$ and $0 \le p \leq 2d-1$, the decision problem whether $\gamma_d^p(G) \leq k$ holds is NP-complete for bipartite planar graphs. For cycles $C_n$, the value $\gamma_d^p(C_n)$ was determined in all cases, while for trees $T$, (in most cases sharp) lower and upper bounds proved for $\gamma_2^0(T)$, $\gamma_2^2(T)$, and $\gamma_d^2(T)$, $d\ge 2$. In~\cite{BIKZ-2025b} it was established that $\gamma_d^1(G) \leq \frac{n}{d+1}$ holds for any bipartite graph $G$ of order $n \geq d+1 \ge 2$, which proves and extends Beineke and Henning's conjecture from 1994. Moreover, for trees $T$ with $\ell$ leaves, it has been proved that $\gamma_d^1(T) \leq \frac{n-\ell}{d}$ and $\gamma_d^1(T) \leq \frac{n+\ell}{d+2}$, where the latter inequality extends Favaron's theorem from 1992~\cite{Favaron-1992} asserting that $\gamma_1^1(T) \leq \frac{n+\ell}{3}$. 

The strong product is one of four standard graph products and is the subject of constant interest. Among recent studies of this product, we mention the papers dealing with coloring aspects~\cite{babu-2024, enomoto-2023, esperet-2024}, the bootstrap percolation~\cite{bresar-2024}, the vertex forwarding index~\cite{qian-2025}, and the clique immersion~\cite{wu-2025}. The independent domination number of the strong product of cycles was studied in~\cite{yang-2018}. Recall that the {\em strong product} $G\boxtimes H$ of graphs $G$ and $H$ has vertex set $V(G)\times V(H)$, while vertices $(g,h)$ and $(g',h')$ are adjacent if one of the following three conditions holds: (i) $gg'\in E(G)$ and $h=h'$, (ii) $g=g'$ and $hh'\in E(H)$, (iii) $gg'\in E(G)$ and $hh'\in E(H)$. If $g\in V(G)$, then the subgraph of $G\boxtimes H$ induced by $\{g\}\times V(H)$ is isomorphic to $H$ and denoted by $^gH$. Similarly, the subgraph induced by $V(G)\times \{h \}$ is isomorphic to $G$ and denoted by $G^h$. These subgraphs $^gH$ and $G^h$ are called {\em layers} of $G\boxtimes H$. By a {\em strong grid} we mean the strong product of two paths, by a {\em strong prism} the strong product of a path by a cycle, and by a {\em strong torus} the strong product of two cycles. The key property of the strong product for us is that $d_{G\boxtimes H}((g,h), (g',h')) = \max\{d_G(g,g'), d_H(h,h')\}$, cf.~\cite[Proposition 5.4]{HIK-2011}. Another basic property of the strong product is that the operation is commutative, hence in all our results the role of factors can be interchanged. For additional information on the many aspects of the strong product we refer to the book~\cite{HIK-2011}.

In this paper we investigate the $d$-distance $p$-packing domination number of the strong product of graphs. In the next section we give additional definitions, recall needed results, and prove that $\gamma_d^p(G\boxtimes H) \le \gamma_d^p(G) \gamma_d^p(H)$ holds for any graphs $G$ and $H$. In Section~\ref{sec:conjecture} we conjecture that if $\gamma_d^p(G) = \infty$, then $\gamma_d^p(G\boxtimes H) = \infty$ for every graph $H$. Several results are proved which support the conjecture. In particular, the conjecture holds true if $H$ contains a so-called $(d,p)$-close vertex, which represents a new concept that can receive independent attention. Also, we prove that if $\gamma_d^p(C_m)= \infty$, then $\gamma_d^p(C_m \boxtimes C_n)=\infty$. In Section~\ref{sec:grids-prisms} we demonstrate that $\gamma_d^p(P_m \boxtimes P_n) = \left \lceil \frac{m}{2d+1} \right \rceil \left \lceil \frac{n}{2d+1} \right \rceil$, and that if $\gamma_d^p(C_n) < \infty$, then $\gamma_d^p(P_m \boxtimes C_n) = \left \lceil \frac{m}{2d+1} \right \rceil \left \lceil \frac{n}{2d+1} \right \rceil$. These results show that the above general upper bound is sharp. On the other hand, in Section~\ref{sec:toruses}, infinite families of strong toruses are presented for which the strict inequality holds. In particular, if $t= 0$ or $t\ge 2$, and $d = \left\lfloor 5(t+1)/2\right\rfloor$, $p=3t+2$, then $\gamma_{d}^{p} \left(C_{11(t+1)}\boxtimes C_{11(t+1)}\right) \le  \gamma_{d}^{p}\left(C_{11(t+1)}\right) \gamma_{d}^{p}\left(C_{11(t+1)}\right) -2$. But certain strong toruses also reach the upper bound, for example, 
if $n\equiv 0 \pmod{2d+1}$, then $ \gamma_d^p(C_m \boxtimes C_n) = \gamma_d^p(C_m)\, \gamma_d^p(C_n)$. In the concluding section we present another general non-sharpness example of the general upper bound, pose some problems, and show that our findings imply that most of the main results from~\cite{anand-2025} are fundamentally wrong. 

\section{Preliminaries and a general bound}
\label{sec:prelim}

In this section, we first provide additional definitions needed in this article. After that we recall some earlier results that will be useful later and give an upper bound on the $d$-distance $p$-packing domination number of arbitrary strong products. 

For two integers $a \le b$, the interval of integers is defined and denoted as $[a,b]=\{a, \dots, b\}$. If $a=b$, then $[a,b] =\{a\}$. The interval $[1,a]$ will be abbreviated to $[a]$.

Let $G$ be a connected graph. The {\em eccentricity} $\ecc_G(u)$ of the vertex $u\in V(G)$ is the maximum distance between $u$ and the remaining vertices of $G$. The minimum and the maximum eccentricity of the vertices of $G$ are, respectively, the {\em radius} $\rad(G)$ and the {\em diameter} $\diam(G)$ of $G$. A {\em pendant path of order $k$} in $G$ is a path subgraph $v_1\dots v_k$ so that $\deg_G(v_1)=1$ and $\deg_G(v_i)=2$ for $i \in [2,k]$. If $u\in V(G)$ and $d\ge 1$, then by $N_G^d[u]$ we denote the $d$-neighborhood of $u$ in $G$, that is, the set of vertices which are at distance at most $d$ from $u$. Further, for $S\subseteq V(G)$ let $N_G^d[S] = \bigcup_{u\in S} N_G^d[u]$.  Unless stated otherwise, we will assume that $V(P_n) = V(C_n) = [n]$, with the natural adjacency. 

From the stock of known facts, let us first recall the following. 

\begin{proposition} {\rm \cite[Proposition 1.1]{BIKZ-2025a}}
\label{prop:lower-bound-k-domination}
If $0 \leq d'\leq d$ and $0\le p\le p'$, then $\gamma_{d'}^{p'}(G) \geq \gamma_d^p(G)$.
\end{proposition}

As it was explicitly or implicitly observed in previous papers, the definition of a $d$-distance $p$-packing dominating set implies the following properties. To be self-sufficient, we also provide arguments for them. 

\begin{obs} \label{obs:infinity}
    For every graph $G$ and every two nonnegative integers $d$ and $p$, the following statements are true.
    \begin{itemize}
        \item[(i)] If $\rad(G) \leq d$, then $\gamma_d^p(G) =1$.
        \item[(ii)] If $\rad(G) > d$ and $p \ge 2d+1$, then $\gamma_d^p(G) =\infty$.
        \item[(iii)] If $p \leq d$, then $\gamma_d^p(G) <\infty$.        
    \end{itemize}    
\end{obs}

\begin{proof}
(i) A vertex $u$ with $\ecc(u) = \rad(G)$ forms a $d$-distance dominating set and a $p$-packing set. 

(ii) Suppose $S$ is a $d$-distance dominating set and a $p$-packing set. Let $x\in S$. Since $\rad(G) > d$, there exists a vertex $y$ with $d_G(x,y) = d+1$. As $S$ is a $d$-distance dominating set, there exists $z\in S$ such that $d_G(z,y) \le d$. But then $d_G(x,z)\le 2d+1$, a contradiction. 

(iii) First set $S = \{u\}$, where $u$ is an arbitrary vertex of $G$. Then, inductively, if there exists a vertex $x$ which is at distance at least $d+1$ to every vertex of $S$, then add $x$ to $S$. After no such addition is possible, the condition $p\le d$ implies that the constructed set is a $d$-distance dominating set and a $p$-packing set.
\end{proof}

For a path $P_n$, Observation~\ref{obs:infinity}~(i) and (ii) imply that $\gamma_d^p(P_n)=\infty$ when $p \ge 2d+1$ and $n \ge 2d+2$, while $\gamma_d^p(P_n) = 1$ when $n \le 2d+1$.  In the remaining cases we have the following result for paths.

\begin{proposition}  {\rm \cite[Proposition 3.3]{BIKZ-2025a}}
\label{prop:path}
    For every three integers $d$, $p$, and $n$ with $0 \le p \le 2d$, it holds that $\gamma_d^p(P_n)= \left\lceil\frac{n}{2d+1}\right\rceil$.
\end{proposition}

The corresponding known result for cycles reads as follows. 

\begin{theorem} {\rm \cite[Theorem 3.1]{BIKZ-2025a}}
\label{thm:cycles}
If $d$, $p$, and $n\ge 3$ are integers with  $0 \leq p \leq 2d$ and $n \ge 2d+2$, then 
 $$
    \gamma_d^p(C_n)=
    \begin{cases}
		\left\lceil\frac{n}{2d+1}\right\rceil; & \frac{n}{p+1} \ge \left\lceil\frac{n}{2d+1} \right\rceil,\\
        	\infty; & \mbox{otherwise}.\\
    \end{cases}
    $$    
\end{theorem}

We conclude this section with the following announced bound. 

\begin{theorem}
\label{thm:upper-for-strong}
If $G$ and $H$ are graphs, and $d$ and $p$ are integers, then 
$$\gamma_d^p(G\boxtimes H) \le \gamma_d^p(G) \gamma_d^p(H)\,.$$ 
\end{theorem}

\begin{proof}
If $\gamma_d^p(G) = \infty$ or $\gamma_d^p(H) = \infty$, then there is nothing to prove. Hence assume in the rest that $\gamma_d^p(G) <\infty$ and $\gamma_d^p(H) <\infty$. 

Let $S_G$ be a $\gamma_d^p$-set of $G$ and $S_H$ a $\gamma_d^p$-set of $H$. To prove the inequality it suffices to verify that $S = S_G\times S_H$ is a $d$-distance $p$-packing set of $G\boxtimes H$. If $(g,h)$ and $(g',h')$ are two vertices from $S$, then having in mind that $S_G$ and $S_H$ are $p$-packings, we have
$$d_{G\boxtimes H}((g,h), (g',h')) = \max \{d_G(g,g'), d_H(h,h')\} \ge \max \{p+1,p+1\} = p+1\,,$$
hence $S$ is a $p$-packing. Let now $(x,y)$ be an arbitrary vertex from $V(G\boxtimes H)\setminus S$. Since $S_G$ is a $d$-distance dominating set of $G$, there exists $g\in S_G$ such that $d_G(x,g)\le d$. Similarly, there exists a vertex $h\in S_H$ such that $d_H(y,h)\le d$. Therefore, $$d_{G\boxtimes H}((x,y), (g,h)) = \max \{d_G(x,g), d_H(y,h)\} \le \max \{d,d\} = d\,,$$
which implies that $S$ is a $d$-distance dominating set. 
\end{proof}

\section{A conjecture and supporting results}
\label{sec:conjecture}

In this section we investigate the question how the assumption $\gamma_d^p(G) = \infty$ affects the strong product with $G$ being its factor. In this regard, we make the following: 

\begin{conjecture}
\label{con:infinite}    
If $G$ is a graph, and $d$ and $p$ are integers such that  $\gamma_d^p(G) = \infty$, then $\gamma_d^p(G\boxtimes H) = \infty$ for every graph $H$.
\end{conjecture}

\begin{remark} \label{other-direction-of-conjecture}
By Theorem~\ref{thm:upper-for-strong} we know that the reverse of the implication in the conjecture is true. That is, if $\gamma_d^p(G\boxtimes H) = \infty$, then $\gamma_d^p(G) = \infty$ or $\gamma_d^p(H) = \infty$. If Conjecture~\ref{con:infinite} is true, then for any $n$-fold strong product $F=G_1 \boxtimes \cdots \boxtimes G_n$ it holds that $\gamma_d^p(F)=\infty $ if and only if $\gamma_d^p(G_i)=\infty $ is true for at least one factor. 
\end{remark}

A {\em perfect $r$-code} of a graph $G$ is a set $X\subseteq V(G)$ such that the sets $N_G^r[u]$, $u\in X$, form a partition of $V(G)$. By definition, the existence of a perfect $r$-code in $G$ is equivalent to the fact that $\gamma_r^{2r}(G) < \infty$. It is proved in~\cite{Abay-2009} that the $n$-fold strong product of simple graphs has a perfect $r$-code if and only if each factor has a perfect $r$-code. This result implies that Conjecture~\ref{con:infinite} is true for $p=2d$. By Observation~\ref{obs:infinity} (i) and (ii), the statement is true when $p \ge 2d+1$. Further, the condition $\gamma_d^p(G) = \infty$ does not hold if $p \leq d$. Therefore, it is enough to consider Conjecture~\ref{con:infinite} for $d <p<2d$. 

A vertex $u$ of a graph $G$ is {\em $(d,p)$-close} if $d_G(x,y) \leq p$ holds for every $x, y\in N_G^d[u]$. We next verify that Conjecture~\ref{con:infinite} holds if $H$ contains a $(d,p)$-close vertex. 

\begin{theorem} \label{thm:pendant-path-infty-neighborhood}
   Let $d$ and $p$ be two nonnegative integers with $p < 2d$ and let $G$ and $H$ be two graphs. If $\gamma_d^p(G)= \infty$ and $H$ contains a $(d,p)$-close vertex, then $\gamma_d^p(G \boxtimes H)=\infty$.  
\end{theorem}
\begin{proof}
Let $\gamma_d^p(G)= \infty$. Observation~\ref{obs:infinity} (iii) then implies $d < p$.
Let $u\in V(H)$ be a $(d,p)$-close vertex. Suppose, to the contrary of the statement, that $\gamma_d^p(G \boxtimes H) < \infty$ and $S$ is a $\gamma_d^p$-set in $G \boxtimes H$.
\medskip

Consider the layer $G^{u}$ in $G \boxtimes H$ and let $X = N_H^d[u]$. Then, $V(G) \times X$ is the set of vertices in the strong product that can $d$-distance dominate $G^{u}$. Projecting the vertices in $S \cap (V(G) \times X)$ to the layer $G^{u}$, we obtain a $d$-distance dominating set of $G^{u}$. By assumption, $\gamma_d^p(G^{u})= \infty$ and hence, the projected vertices cannot form a $p$-packing. It means that there exist two vertices $(x,y)$ and $(x',y')$ in $S\cap (V(G) \times X)$ such that the projected vertices $(x,u)$ and $(x',u)$ satisfy
\begin{equation} \label{eq:00}
  p \ge d_{G\boxtimes H} ((x,u), (x',u))= d_G(x,x').  
\end{equation} 
Since $S$ is a $p$-packing, we have
\begin{equation} \label{eq:0}
    p+1 \le d_{G\boxtimes H} ((x,y), (x',y'))= \max \{d_G(x,x'), d_H(y,y')\}.
\end{equation}
Since $y$ and $y'$ belong to $X = N_H^d[u]$, and $u$ is a $(d,p)$-close vertex, we have $d_H(y,y') \leq p$. This inequality, together with~\eqref{eq:00}, contradicts~\eqref{eq:0}. We may conclude $\gamma_d^p(G \boxtimes H) = \infty$.
\end{proof}

Complete graphs and graphs with diameter at most $p\ge 1$ contain $(d,p)$-close vertices. If $p=2d-1$ and $H$ contains a simplicial vertex (i.e., a vertex with a complete neighborhood), then again, $\gamma_d^p(G)= \infty$ implies $\gamma_d^p(G \boxtimes H)=\infty$. In addition, if $H\colon z_1 \dots z_n$ is a path, then the leaf $z_1$ is a $(d,p)$-close vertex. Indeed, if $d_H(z_1,z_i) \leq d$ and $d_H(z_1,z_j) \leq d$, then $d_H(z_i,z_j) \leq d \le p$. Theorem~\ref{thm:pendant-path-infty-neighborhood} then directly implies the following statement:

\begin{corollary} \label{cor:path-infty}
 Let $G$ be a graph and $d$, $p$, $n$ three nonnegative integers with $  p < 2d$ and $2 \le n$.  If $\gamma_d^p(G)= \infty$, then $\gamma_d^p(G \boxtimes P_n)=\infty$.  
\end{corollary}
Assuming that $H$ contains a pendant path that is long enough, a more general consequence of Theorem~\ref{thm:pendant-path-infty-neighborhood} can be proved. 
\begin{corollary}   \label{cor:pendant-path-infty}
   Let $d$ and $p$ be two nonnegative integers with $p < 2d$ and let $G$ and $H$ be two graphs. If $\gamma_d^p(G)= \infty$ and $H$ contains a pendant path on at least $\frac{2d-p}{2}$ vertices, then $\gamma_d^p(G \boxtimes H)=\infty$.  
\end{corollary} 

\begin{proof}
 Let $k= \lceil \frac{2d-p}{2} \rceil$ and let $z_1z_2\dots z_k$ be a pendant path in $H$ such that $z_k$ is adjacent to a vertex $z \in V(H)$ different from $z_{k-1}$. Observe that $d_H(z_1, z)=k \ge \frac{2d-p}{2}$. In light of Theorem~\ref{thm:pendant-path-infty-neighborhood} it suffices to prove that $z_1$ is a $(d,p)$-close vertex. 

If $v_i$ lies on a shortest path between $z_1$ and $v_{3-i}$, for $i \in \{1,2\}$, then $d_H(v_1,v_2) \leq d_H(z_1, v_{3-i}) \leq d <p$ proves the statement. Otherwise, neither $v_1$ nor $v_2$ belongs to the pendant path $z_1 \dots z_k$, and the distance can be estimated as
\begin{align*}
d_H(v_1,v_2) & \leq d_H(z, v_1)+d_H(z,v_2) = (d_H(z_1, v_1)-k)+(d_H(z_1,v_2)-k) \\
& \le (d-k) + (d-k) \leq 2d- (2d-p)=p.
\end{align*}
This finishes the proof of the statement.  
\end{proof}
  
To further support Conjecture~\ref{con:infinite}, we prove the following result. 

\begin{theorem} \label{thm:strong-torus-infinity}
For every two nonnegative integers $d$ and $p$ and for every two cycles $C_m$ and $C_n$, if $\gamma_d^p(C_m)= \infty$, then $\gamma_d^p(C_m \boxtimes C_n)=\infty$.    
\end{theorem}
\proof
If $p \ge 2d+1$, the statement follows by Observation~\ref{obs:infinity}~(ii). If $p=2d$ then, under the present conditions, the nonexistence of a perfect $d$-code in $C_m \boxtimes C_n$ follows by~\cite{Abay-2009}. By Observation~\ref{obs:infinity}~(iii), the condition $\gamma_d^p(C_m) = \infty$ does not hold if $p \leq d$. Therefore, it suffices to consider the cases with $d < p <2d$.

The vertex set of the layer $(C_m)^i$ is denoted by $R_i$, for $i \in [n]$; that is $R_i = \{(x,i) \colon x \in [m]\}.$ Sometimes the $m$-cycle induced by $R_i$ is also referred to as $R_i$. 

By Theorem~\ref{thm:cycles}, the condition $\gamma_d^p(C_m)= \infty $ implies $\frac{m}{p+1} < \lceil \frac{m}{2d+1} \rceil$. Since the assumption $p <2d$ gives $\frac{m}{p+1} >  \frac{m}{2d+1} $, we may infer that 
\begin{equation} \label{eq:1}
\left\lfloor \frac{m}{p+1} \right\rfloor =  \left\lfloor \frac{m}{2d+1} \right\rfloor <
     \frac{m}{2d+1}.  
\end{equation}
The inequality $\frac{m}{p+1}- \frac{m}{2d+1} <1 $ also follows and it is equivalent to the following inequality to be applied later:
\begin{equation} \label{eq:2}
    \frac{m}{2d+1} < \frac{p+1}{2d-p}.
\end{equation}

Consider a row $R_i$ in $C_m \boxtimes C_n$ and define $A_i=\bigcup_{j=i-d}^{i+d} R_j$ and $B_i= \bigcup_{j=i-d}^{i-d+p} R_j$ for every $i \in [n]$. Thus $A_i$,  $B_i$, and $A_i \setminus B_i$ are the union of $2d+1$ rows, $p+1$ rows, and $2d-p$ rows, respectively. See Fig.~\ref{fig:sets-in-product-of-cycles}.

\begin{figure}[ht!]
\begin{center}
\begin{tikzpicture}[scale=1.0,style=thick,x=1cm,y=1cm]
\def\vr{3pt}
\begin{scope}[xshift=-0cm, yshift=0cm] 
\coordinate(x1) at (1.0,-0.5);
\coordinate(x2) at (2.0,-0.5);
\coordinate(xm) at (10.0,-0.5);
\coordinate(xm-1) at (9.0,-0.5);
\coordinate(y1) at (-0.5,1.0);
\coordinate(y2) at (-0.5,1.7);
\coordinate(yn) at (-0.5,8.2);
\coordinate(yn-1) at (-0.5,7.5);
\coordinate(yi) at (-0.5,5.0);
\coordinate(z1) at (1.1,5.0);
\coordinate(z2) at (2.0,5.0);
\coordinate(zm) at (9.9,5.0);
\coordinate(zm-1) at (9.0,5.0);
\draw (-1,0) -- (11.5,0);
\draw (0,-1) -- (0,9);
\draw (x1) -- (x2);
\draw (x2) -- (2.5,-0.5);
\draw (xm) -- (xm-1);
\draw (7.5,-0.5) -- (xm-1);
\draw (1,-0.5) .. controls (2,0) and (9,0) .. (10,-0.5);
\draw (y1) -- (y2);
\draw (y2) -- (-0.5,2.2);
\draw (yn) -- (yn-1);
\draw (-0.5,7.1) -- (yn-1);
\draw (y1) .. controls (0,2) and (0,7.2) .. (yn);
\draw (yi) -- (-0.5,5.3);
\draw (yi) -- (-0.5,4.7);
\draw [draw=black] (0.5,0.5) rectangle (10.5,8.5);
\draw [draw=black, dashed] (0.7,3) rectangle (10.3,7);
\draw [stealth-stealth](10.7,3.0) -- (10.7,7.0);
\draw [draw=black, densely dotted] (0.8,3.1) rectangle (10.2,6.0);
\draw [stealth-stealth](5.5,3.1) -- (5.5,6.0);
\draw [draw=black] (0.9,4.75) rectangle (10.1,5.25);
\draw [stealth-stealth](6.5,6.0) -- (6.5,7.0);
\draw(x1)[fill=white] circle(\vr);
\draw(x2)[fill=white] circle(\vr);
\draw(xm)[fill=white] circle(\vr);
\draw(xm-1)[fill=white] circle(\vr);
\draw(y1)[fill=white] circle(\vr);
\draw(y2)[fill=white] circle(\vr);
\draw(yn)[fill=white] circle(\vr);
\draw(yn-1)[fill=white] circle(\vr);
\draw(yi)[fill=white] circle(\vr);
\draw(z1)[fill=white] circle(\vr);
\draw(z2)[fill=white] circle(\vr);
\draw(zm)[fill=white] circle(\vr);
\draw(zm-1)[fill=white] circle(\vr);

\node [below=1.0mm] at (x1) {$1$};
\node [below=1.0mm] at (x2) {$2$};
\node [below=1.0mm] at (xm-1) {$m-1$};
\node [below=1.0mm] at (xm) {$m$};
\node [right=5.0mm] at (xm) {$C_m$};
\node [right=7.0mm] at (x2) {$\dots$};
\node [left=7.0mm] at (xm-1) {$\dots$};
\node [left=1.0mm] at (y1) {$1$};
\node [left=1.0mm] at (y2) {$2$};
\node [left=1.0mm] at (yn-1) {$n-1$};
\node [left=1.0mm] at (yn) {$n$};
\node [left=1.0mm] at (yi) {$i$};
\node [above=3.0mm] at (yn) {$C_n$};
\node [above=5.0mm] at (y2) {$\vdots$};
\node [below=2.0mm] at (yn-1) {$\vdots$};
\node at (9.5,8) {$C_m\boxtimes C_n$};
\node [right=7.0mm] at (z2) {$\dots$};
\node [left=7.0mm] at (zm-1) {$\dots$};
\node at (11.4, 5.0) {$2d+1$};
\node at (6.1, 4.0) {$p+1$};
\node at (7.2, 6.5) {$2d-p$};
\node at (9.8, 6.6) {$A_i$};
\node at (1.2, 3.5) {$B_i$};
\node at (4.5, 5) {$R_i$};
\end{scope}
\end{tikzpicture}
\caption{Sets $R_i$, $A_i$, and $B_i$ in $C_m\boxtimes C_n$ with respect to given $d$ and $p$, where $d < p <2d$.}
\label{fig:sets-in-product-of-cycles}
\end{center}
\end{figure}
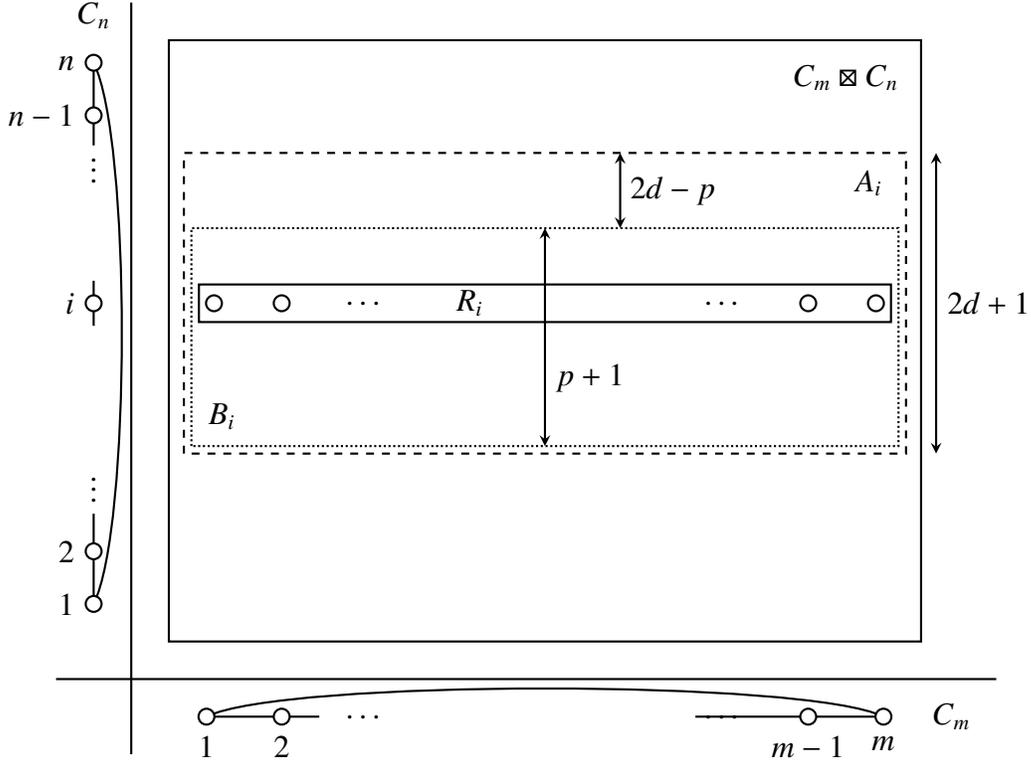

Suppose for a contradiction that $\gamma_d^p(C_m \boxtimes C_n)$ is finite and let $S$ be a $\gamma_d^p$-set in $C_m \boxtimes C_n$. First observe that for every two vertices $v=(x,y)$ and $v'=(x',y')$ from $B_i\cap S$, the difference $|y-y'|$ is at most $p$. Since $S$ is a $p$-packing, $|x-x'| \ge p+1$ follows. Then, projecting the vertices in $B_i \cap S$ to $R_i$, the obtained set is a $p$-packing in the cycle $R_i$. Consequently, $|B_i \cap S| \leq \lfloor \frac{m}{p+1}\rfloor$ holds for each $i \in [n]$. Count now the pairs $(i, v)$ with $i \in [n]$ and $v \in B_i \cap S$. Since each $v \in S$ belongs to exactly $p+1$ different sets $B_i$, the number of such pairs is $|S|\,(p+1)$. On the other hand, we know that every $i \in [n]$ is contained in at most $\lfloor \frac{m}{p+1}\rfloor$ pairs. Therefore,
\begin{equation} \label{eq:3}
    |S|\, (p+1) \leq n \left\lfloor \frac{m}{p+1}\right\rfloor.
\end{equation}

Since $S$ is supposed to be a $\gamma_d^p$-set in the strong torus, every vertex in $R_i$ is $d$-distance dominated by a vertex in $A_i$. If we project the vertices in $A_i \cap S$ to the row $R_i$, we get a $d$-distance dominating set in the cycle $R_i$. 
As $\gamma_d^p(C_m)=\infty$, this set of projected vertices is not a $p$-packing in $R_i$ and hence, not all vertices in $A_i \cap S$ belong to $B_i$. We conclude that $A_i \setminus B_i$ contains at least one vertex from $S$, for every $i \in [n]$. We now count the pairs $(i,u)$ with $i \in [n]$ and $u \in (A_i \setminus B_i) \cap S$. Since every $u \in S$ belongs to $2d-p$ different sets $A_i \setminus B_i$, the number of such pairs is $|S|\,(2d-p)$. As we also know that every $i \in [n]$ belongs to at least one such pair, we conclude
\begin{equation} \label{eq:4}
   |S|\,(2d-p) \ge n. 
\end{equation}
As $p < 2d$, inequalities~\eqref{eq:3} and \eqref{eq:4} imply
\begin{equation} \label{eq:5}
   \frac{n}{2d-p} \leq |S| \leq \frac{n}{p+1} \left\lfloor \frac{m}{p+1}\right\rfloor.
\end{equation}
The above findings can be combined as follows to arrive at a contradiction: 
\begin{equation*}
    \frac{m}{2d+1} \overset{\text{\eqref{eq:2}}}{<}
    \frac{p+1}{2d-p} \overset{\text{\eqref{eq:5}}}{\le}
    \left\lfloor \frac{m}{p+1}\right\rfloor \overset{\text{\eqref{eq:1}}}{=}
    \left\lfloor \frac{m}{2d+1}\right\rfloor.
    \end{equation*}
This proves that no $\gamma_d^p$-set exists in $C_m \boxtimes C_n$ and therefore $\gamma_d^p(C_m \boxtimes C_n)= \infty$ as stated. 
\qed

\section{Strong grids and prisms}
\label{sec:grids-prisms}

In this section we determine the $d$-distance $p$-packing domination number of strong grids and of strong prisms. These results demonstrate sharpness of bound of Theorem~\ref{thm:upper-for-strong}. 

\begin{theorem}
    \label{thm:equality-grids}
    If $p \leq 2d$, then $\gamma_d^p(P_m \boxtimes P_n) = \left \lceil \frac{m}{2d+1} \right \rceil \left \lceil \frac{n}{2d+1} \right \rceil$, and if $\gamma_d^p(C_n) < \infty$, then $\gamma_d^p(P_m \boxtimes C_n) = \left \lceil \frac{m}{2d+1} \right \rceil \left \lceil \frac{n}{2d+1} \right \rceil$.
\end{theorem}

\begin{proof}
    Let $n,m\ge 2$ and set $G = P_m\boxtimes P_n$ for the rest of the proof. We claim that $\gamma_d^p(G) \geq \gamma_d^p(P_m) \gamma_d^p (P_n) = \lceil \frac{m}{2d+1}\rceil \cdot \lceil \frac{n}{2d+1}\rceil$. (The other inequality follows from Theorem~\ref{thm:upper-for-strong}.) By Proposition~\ref{prop:lower-bound-k-domination}, $\gamma_d^p(G) \geq \gamma_d^0(G)$, thus $\gamma_d^p(G) \geq \gamma_d(G)$. Therefore it suffices to prove $\gamma_d(G) \geq \left \lceil \frac{m}{2d+1} \right \rceil \left \lceil \frac{n}{2d+1} \right \rceil$. Recall that $\gamma_d(G) = \gamma_d^0(G)$. 
    
    We proceed by induction on $m$. If $m\le 2d+1$, then the projection of a minimum $d$-distance dominating set of $G$ to the first layer of $P_n$ is a $d$-distance dominating set of $P_n$. Hence $\gamma_d(G) \ge \gamma_d(P_n)$ when $m\le 2d+1$. 

    Assume in the rest that $m\ge 2d+2$. Let $S$ be a minimum $d$-distance dominating set of $G$, and let $G_1$, $G_2$, and $G_3$ be subgraphs of $G$ induced by $[d+1]\times V(P_n)$, by $[d+2, 2d+1]\times V(P_n)$, and by $[2d+2,m] \times V(P_n)$, respectively. Note first that $|S\cap V(G_1)|\ \ge \left\lceil \frac{n}{2d+1} \right\rceil$. Indeed,  otherwise there would exist a vertex in $\{1\} \times V(P_n)$ at distance at least $d+1$ from every vertex of $S$, which is not possible as $S$ is a $d$-distance dominating set. We now distinguish two cases. 
    
    \medskip\noindent
    {\bf Case 1}: $S\cap V(G_2) = \emptyset$. \\
    In this case, $S\cap V(G_3)$ is a $d$-distance dominating set of $G_3$. Indeed, $d_G(x_3, x_1) \geq d+1$ for every $x_1 \in V(G_1)$ and $x_3 \in V(G_3)$, thus vertices from $S \cap V(G_1)$ cannot $d$-distance dominate any vertex in $G_3$. Hence, using the induction hypothesis, we get $|S\cap V(G_3)| \ge \gamma_d(G_3) = \left\lceil \frac{n}{2d+1}\right\rceil \cdot \left\lceil \frac{m-(2d+1)}{2d+1}\right\rceil$. Therefore, 
    \begin{align*}
    |S| & = |S\cap V(G_1)| + |S\cap (G_3)| \\
    & \ge \left\lceil \frac{n}{2d+1}\right\rceil + \left\lceil \frac{n}{2d+1}\right\rceil\cdot \left\lceil \frac{m-(2d+1)}{2d+1}\right\rceil = \left\lceil \frac{n}{2d+1}\right\rceil\cdot \left\lceil \frac{m}{2d+1}\right\rceil\,,    
    \end{align*}
    which settles this case.
    
    \medskip\noindent
    {\bf Case 2}: $S\cap V(G_2) \ne \emptyset$. \\
    The idea in this case is to bound the size of $S$ by constructing a $d$-distance dominating set of $G_3$ of size $|S \cap V(G_2)| + |S \cap V(G_3)|$. Suppose that $S \cap V(G_2)=\{v_1, \dots , v_\ell\}$. First, set $Z_0=S \cap V(G_3)$ and consider those vertices $(g,h)$ in $G_3$ which satisfy $d_{G}((g,h), v_1)\le d$ but  $d_{G}((g,h), (g',h'))\ge d+1$ holds for each vertex $(g',h')\in Z_0$. Let $Y_1$ be the set of these vertices. If $(g,h) \in Y_1$, then $g \in [2d+2, 3d+1]$.
    Moreover, if $(g,h) \in Y_1$ for some $g \in [2d+2,3d+1]$, 
    then $(2d+2,h) \in Y_1$. 
    
    We claim that there is a vertex $v_1^*= (2d+2, h^*)$ in $Y_1$ such that $N_G^d[v_1^*]$ covers $Y_1$. Indeed, suppose that $\{(2d+2,x),(2d+2,x+1)\} \subseteq N_G^d[v_1] $, $(2d+2,x) \in Y_1$, and $(2d+2,x+1) \notin Y_1$. Then, there exists a vertex $(g,h) \in Z_0$ such that $N_G^d[(g,h)]$ covers only $(2d+2,x+1)$. Therefore, we have $2d+2 \le g \le 3d+2$ and $h=x+d+1$, and consequently, $N_G^d[(g,h)]$ covers all vertices from $Y_1$ with second entry at least $x+1$. Similarly, if $\{(2d+2,y),(2d+2,y-1)\} \subseteq N_G^d[v_1]$, $(2d+2,y) \in Y_1$, and $(2d+2,y-1) \notin Y_1 $, then $Z_0$ contains a vertex $(g,h)$ with $2d+2 \le g \le 3d+2$ and $h=y-d-1$. We may infer that $Y_1 \cap V(^{2d+2}P_n)$ contains consecutive vertices from the layer $^{2d+2}P_n$ and therefore, we can choose a middle vertex $v_1^*$ from this intersection such that $N_G^d[v_1^*]$ covers $Y_1 \cap V(^{2d+2}P_n)$. Then, in turn, $N_G^d[v_1^*]$ covers the entire $Y_1$. 
    Let us now define $Z_1= Z_0 \cup \{v_1^*\}$. By the choice of $v_1^*$, $Y_1 \cup N_G^d[Z_0] \subseteq N_G^d[Z_1]$ holds.
    
    After a set $Z_i$ is obtained, for $1 \leq i < \ell$, we repeat the process with $v_{i+1} \in S \cap V(G_2) $ and $Z_i$. That is, we define $Y_{i+1}= (N_G^d[v_{i+1}] \setminus N_G^d[Z_i]) \cap V(G_3)$ and determine a vertex $v_{i+1}^*=(2d+2,h)$ in $Y_{i+1}$ such that $Z_{i+1}= Z_i \cup \{v_{i+1}^*\}$ and  $Y_{i+1} \cup N_G^d[Z_i] \subseteq N_G^d[Z_{i+1}]$. If $Y_j$ is empty in a step, we may proceed with $Z_j= Z_{j-1}$, but the occurrence of this situation would contradict the minimality of $S$. 
    
    At the end of the process, we have a set $Z_\ell$ in $G_3$ that contains at most $|S \cap V(G_2)|+ |S \cap V(G_3)|$ vertices. Moreover, $Z_\ell $ is a $d$-distance dominating set of $G_3$ as any vertex $u$ from $G_3$ was either covered by a vertex from $Z_0= S \cap V(G_3)$ or by a vertex $v_j \in S \cap V(G_2)$. In the latter case, $u \in N_G^d[v_j^*]$ and hence, $u \in N_G^d[Z_\ell]$ in both cases. By applying the induction hypothesis to $G_3$, we obtain
    $$|S\cap V(G_2)| + |S\cap V(G_3)| \ge  |Z_\ell| \ge \gamma_d(G_3) = \left\lceil \frac{n}{2d+1}\right\rceil\cdot \left\lceil \frac{m-(2d+1)}{2d+1}\right\rceil.
    $$
    To finish the proof, we recall that $|S\cap V(G_1)|\ \ge \left\lceil \frac{n}{2d+1} \right\rceil$, henceforth the conclusion for Case~2 is obtained as 
    \begin{align*}
                \gamma_d(G) & = |S| = |S\cap V(G_1)| + |S\cap V(G_2)| + |S\cap V(G_3)|\\
                 & \ge \left\lceil \frac{n}{2d+1}\right\rceil +  \left\lceil \frac{n}{2d+1}\right\rceil\cdot \left\lceil \frac{m-(2d+1)}{2d+1}\right\rceil =
                 \left\lceil \frac{n}{2d+1}\right\rceil\cdot \left\lceil \frac{m}{2d+1}\right\rceil.
                       \end{align*}
    This completes the proof for strong grids. An analogous argument shows that $\gamma_d^p(P_m \boxtimes C_n)= \gamma_d^p(P_m)\, \gamma_d^p(C_n) $ holds for all $m \ge 1$ and $n \ge 3$ provided that $\gamma_d^p(C_n)$ is finite. That is, for a fixed cycle $C_n$, we proceed by induction on $m$ just as above, counting is considered modulo $n$ when $v_i^*$ is determined, and taking account of the fact that $\gamma_d^p(C_n) = \left\lceil \frac{n}{2d+1} \right\rceil$. 
\end{proof}


\section{Strong toruses}
\label{sec:toruses}

In the previous section we have seen that the upper bound of Theorem~\ref{thm:upper-for-strong} is sharp for strong grids and strong prisms.  In this section we demonstrate that the equality does not hold in general and a bit surprisingly this happens already from strong products of cycles. Our first related result reads as follows. 

\begin{theorem}
\label{thm:C11(t+1)}
If $t= 0$ or $t\ge 2$, and $d=\left\lfloor \frac{5}{2}(t+1)\right\rfloor$, $p=3t+2$, then $$\gamma_{d}^{p}\left(C_{11(t+1)}\boxtimes C_{11(t+1)}\right) \le 7= \gamma_{d}^{p}\left(C_{11(t+1)}\right) \gamma_{d}^{p}\left(C_{11(t+1)}\right) -2.
$$
\end{theorem}

\begin{proof}
For a simpler discussion, here we consider every $n$-cycle with the vertex set $V(C_n)= [0,n-1]$ instead of $[n]$. 
We also define $G_t = C_{11(t+1)}\boxtimes C_{11(t+1)}$, for $t=0$ and $t\ge 2$. 

Consider first $G_0$, and let $X = \{(2,0), (3,3), (4,6), (7,9), (8,1), (9,4), (10,7)\}$, see Fig.~\ref{fig:C11}. 

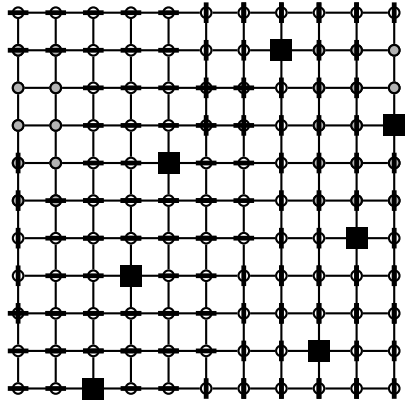
\begin{figure}[ht!]
    \begin{center}
        \begin{tikzpicture}[thick,scale=0.5]
        \tikzstyle{vertex}=[circle, draw, fill=black!10, inner sep=0pt, minimum width=4pt]
        \tikzstyle{vertexBlack}=[rectangle, draw, fill=black, inner sep=0pt, minimum width=7pt]
        \tikzstyle{vertexWhite}=[circle, draw, fill=white, inner sep=0pt, minimum width=4pt]
        \tikzstyle{vertexGray}=[circle, draw, fill=black!30, inner sep=0pt, minimum width=4pt]
        \tikzstyle{vertexBlack2}=[rectangle, rotate=-90, draw, fill=black, inner sep=0pt, minimum width=7pt]
        
        \pgfmathtruncatemacro{\N}{5}
        \pgfmathtruncatemacro{\R}{1}

        \foreach \x in {0,...,10}
            \foreach \y in {0,...,10}
            \node[vertexWhite] (\x+\y) at (\x,\y) {};

        \foreach \y in {0,...,10}
            \foreach \x [remember=\x as \lastx (initially 0)] in {1,...,10}
            \path (\x+\y) edge (\lastx+\y);

        \foreach \x in {0,...,10}
            \foreach \y [remember=\y as \lasty (initially 0)] in {1,...,10}
            \path (\x+\y) edge (\x+\lasty);

        \foreach \x in {8,9,10,0,1}
            \foreach \y in {5,...,9}
                \node[vertexGray] (\x+\y) at (\x,\y) {};

        \foreach \x in {0,...,4}
            \foreach \y in {0,1,2,9,10}
                \node[vertexBlack] (\x+\y) at (\x,\y) {};

        \foreach \x in {1,...,5}
            \foreach \y in {1,...,5}
                \node[vertexBlack] (\x+\y) at (\x,\y) {};

        \foreach \x in {2,...,6}
            \foreach \y in {4,...,8}
                \node[vertexBlack] (\x+\y) at (\x,\y) {};

        \foreach \x in {7,8,9,10,0}
            \foreach \y in {2,...,6}
                \node[vertexBlack2] (\x+\y) at (\x,\y) {};

        \foreach \x in {6,...,10}
            \foreach \y in {0,1,2,3,10}
                \node[vertexBlack2] (\x+\y) at (\x,\y) {};

        \foreach \x in {5,...,9}
            \foreach \y in {7,...,10,0}
                \node[vertexBlack2] (\x+\y) at (\x,\y) {};

        \node[fill=black] (a) at (2,0) {};
        \node[fill=black] (b) at (3,3) {};
        \node[fill=black] (c) at (4,6) {};
        \node[fill=black] (d) at (7,9) {};
        \node[fill=black] (e) at (8,1) {};
        \node[fill=black] (f) at (9,4) {};
        \node[fill=black] (f) at (10,7) {};

        \end{tikzpicture}
        \caption{A schematic representation of $C_{11} \boxtimes C_{11}$; vertices from $X$ are marked with black squares. For clarity, not all edges are drawn. From left to right, the vertices which are $2$-distance dominated by the first three vertices from $X$ are marked with vertical black strips, the vertices which are $2$-distance dominated by the next three with horizontal black strips, and the vertices which are $2$-distance dominated by the final one with color gray.}
        \label{fig:C11}
    \end{center}
\end{figure}

If $t=0$, then $d=p=2$ and it can be checked that $X$ forms a $2$-distance $2$-packing set of $G_0$. It follows that the inequality holds true for $t=0$. For $t\ge 2$, set 
$$X_t = \{(t+1)i, (t+1)j):\ (i,j)\in X\}\,,$$
and consider $X_t$ as a subset of vertices of $G_t$. 

Intuitively, we can think of $X_t$ as a set of vertices obtained from $X$ by adding $t$ additional layers cyclically between each two consecutive horizontal layers and each two vertical consecutive layers of $G_0$. From this description we infer that if $(i,j), (i',j')\in X$, and $d_{G_0}((i,j), (i',j')) = \ell$, then 
$$d_{G_t}((t+1)i, (t+1)j), ((t+1)i', (t+1)j')) = (t+1)\ell\,.$$ 
Since $d_{G_0}((i,j), (i',j'))\ge 3$, it follows that 
$$d_{G_t}((t+1)i, (t+1)j), ((t+1)i', (t+1)j')) \ge 3(t+1)=p+1\,.$$ 
We can conclude that $X_t$ is a $p$-packing of $G_t$. 

Let $(i'',j'')$ be an arbitrary vertex of $G_0$. Then there exists $(i,j)\in X$ such that $d_{G_0}((i,j), (i'',j'')) \le 2$. If follows that 
$$d_{G_t}((t+1)i, (t+1)j), ((t+1)i'', (t+1)j'')) \le 2(t+1)\,.$$ 

Since every vertex of $G_t$ is at distance at most $\lfloor\frac{t+1}{2}\rfloor$ from a vertex of the form $((t+1)i'', (t+1)j''))$, we get that every vertex of $G_t$ is at distance at most $2(t+1) + \lfloor\frac{t+1}{2}\rfloor= d$ from some vertex of $X_t$. Hence $X_t$ is a $d$-distance dominating set, and since $|X_t|=7$, we conclude  
$$\gamma_{d}^{p}\left(C_{11(t+1)}\boxtimes C_{11(t+1)}\right) \le |X_t| =7.
$$
\medskip

For the second part of the statement, we can make the following observations for every $t\neq 1$. If $t$ is odd, then $t \ge 3$ and 
$$ \left\lceil \frac{11t+11}{2d+1} \right\rceil = \left\lceil \frac{11t+11}{5(t+1)+1} \right\rceil =3.
$$
Therefore,
\begin{equation} \label{eq:C11}
    \frac{11t+11}{p+1}= \frac{11t+11}{3t+3} > \left\lceil \frac{11t+11}{2d+1} \right\rceil = 3,
\end{equation}
and Theorem~\ref{thm:cycles} implies 
$\gamma_d^p(C_{11(t+1)})=3$.

Similarly, if $t$ is even, then 
$$ \left\lceil \frac{11t+11}{2d+1} \right\rceil = \left\lceil \frac{11t+11}{(5t+4)+1} \right\rceil =3,
$$
and inequality~\eqref{eq:C11} holds true again.
Theorem~\ref{thm:cycles} then implies 
$\gamma_d^p(C_{11(t+1)})=3$ as before, which finishes the proof of the theorem.
\end{proof}

\noindent
We would like to make the following comments on Theorem~\ref{thm:C11(t+1)}. 
\begin{remark}
The case $t=1$ is special and hence excluded from the theorem because in this case we have $\gamma_5^5(C_{22}) = 2$.

By Theorem~\ref{thm:C11(t+1)}, $\gamma_2^2(C_{11}\boxtimes C_{11})\le 7$. On the other hand, we have checked by computer that no smaller $2$-distance $2$-packing set of $C_{11}\boxtimes C_{11}$ exists, hence we actually have $\gamma_2^2(C_{11}\boxtimes C_{11}) = 7$.

Our conclusion, by Theorem~\ref{thm:C11(t+1)}, is that the strong products $\gamma_{d}^{p}\left(C_{11(t+1)}\boxtimes C_{11(t+1)}\right)$ with $t\neq 1$, $d=\left\lfloor \frac{5}{2}(t+1)\right\rfloor$, and $p=3t+2$ form an infinite family of cases for which the upper bound of Theorem~\ref{thm:upper-for-strong} is not sharp and the difference is at least two.
\end{remark}

Moreover, we have the following infinite family of graphs that shows that the difference $\gamma_2^2(G)  \gamma_2^2(H) - \gamma_2^2(G \boxtimes H)$ can be 2. Again, we take $V(C_n) = [0, n-1]$. Let $N = 55k + 11$, $k \geq 0$. For $C_N \boxtimes C_{11}$, Theorem~\ref{thm:upper-for-strong} gives
$$\gamma_2^2(C_N \boxtimes C_{11}) \leq \left \lceil \frac{11}{5} \right \rceil \left \lceil \frac{55k+11}{5} \right \rceil = 3 (11k+3) = 33k+9.$$ But it turns out that this bound is not attained. Consider 
$$S = \bigcup_{i=1}^{11k+2} \{ (5i-3, 9i-9), (5i-2, 9i-6), (5i-1, 9i-3) \} \cup \{ (N-1, 7) \}.$$
It is not difficult to see that $S$ is a $2$-distance dominating set and also a 2-packing of $C_N \boxtimes C_{11}$ (see Fig.~\ref{fig:family-2}), thus
$$\gamma_2^2(C_N \boxtimes C_{11}) \leq 3(11k+1) + 4 = 33k + 7.$$

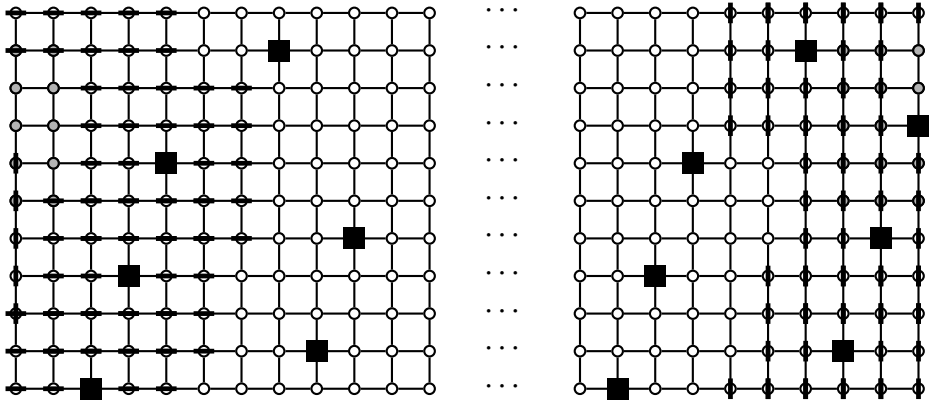
\begin{figure}[!ht]
    \begin{center}
        \begin{tikzpicture}[thick,scale=0.5]
        \tikzstyle{vertex}=[circle, draw, fill=black!10, inner sep=0pt, minimum width=4pt]
        \tikzstyle{vertexBlack}=[rectangle, draw, fill=black, inner sep=0pt, minimum width=7pt]
        \tikzstyle{vertexWhite}=[circle, draw, fill=white, inner sep=0pt, minimum width=4pt]
        \tikzstyle{vertexGray}=[circle, draw, fill=black!30, inner sep=0pt, minimum width=4pt]
        \tikzstyle{vertexBlack2}=[rectangle, rotate=-90, draw, fill=black, inner sep=0pt, minimum width=7pt]
        
        \pgfmathtruncatemacro{\N}{5}
        \pgfmathtruncatemacro{\R}{1}

        \foreach \x in {0,...,11}
            \foreach \y in {0,...,10}
            \node[vertexWhite] (\x+\y) at (\x,\y) {};

        \foreach \y in {0,...,10}
            \foreach \x [remember=\x as \lastx (initially 0)] in {1,...,11}
            \path (\x+\y) edge (\lastx+\y);

        \foreach \x in {0,...,11}
            \foreach \y [remember=\y as \lasty (initially 0)] in {1,...,10}
            \path (\x+\y) edge (\x+\lasty);

        \foreach \y in {0,...,10}
            \node (13+\y) at (13, \y) {$\cdots$};
            
        \foreach \x in {15,...,24}
            \foreach \y in {0,...,10}
            \node[vertexWhite] (\x+\y) at (\x,\y) {};

        \foreach \y in {0,...,10}
            \foreach \x [remember=\x as \lastx (initially 15)] in {16,...,24}
            \path (\x+\y) edge (\lastx+\y);

        \foreach \x in {15,...,24}
            \foreach \y [remember=\y as \lasty (initially 0)] in {1,...,10}
            \path (\x+\y) edge (\x+\lasty);
        
        
        \foreach \x in {22,23,24,0,1}
            \foreach \y in {5,...,9}
                \node[vertexGray] (\x+\y) at (\x,\y) {};

        \foreach \x in {0,...,4}
            \foreach \y in {0,1,2,9,10}
                \node[vertexBlack] (\x+\y) at (\x,\y) {};

        \foreach \x in {1,...,5}
            \foreach \y in {1,...,5}
                \node[vertexBlack] (\x+\y) at (\x,\y) {};

        \foreach \x in {2,...,6}
            \foreach \y in {4,...,8}
                \node[vertexBlack] (\x+\y) at (\x,\y) {};

        \foreach \x in {21,...,24,0}
            \foreach \y in {2,...,6}
                \node[vertexBlack2] (\x+\y) at (\x,\y) {};

        \foreach \x in {20,...,24}
            \foreach \y in {0,1,2,3,10}
                \node[vertexBlack2] (\x+\y) at (\x,\y) {};

        \foreach \x in {19,...,23}
            \foreach \y in {7,...,10,0}
                \node[vertexBlack2] (\x+\y) at (\x,\y) {};

        \node[fill=black] (a) at (2,0) {};
        \node[fill=black] (b) at (3,3) {};
        \node[fill=black] (c) at (4,6) {};
        \node[fill=black] (d) at (7,9) {};
        \node[fill=black] (e) at (8,1) {};
        \node[fill=black] (f) at (9,4) {};

        \node[fill=black] (g) at (24,7) {};
        \node[fill=black] (h) at (23,4) {};
        \node[fill=black] (i) at (22,1) {};
        \node[fill=black] (j) at (21,9) {};
        \node[fill=black] (k) at (18,6) {};
        \node[fill=black] (l) at (17,3) {};
        \node[fill=black] (m) at (16,0) {};

        \end{tikzpicture}
        \caption{A schematic representation of $C_N \boxtimes C_{11}$; vertices from $S$ are marked with black squares. For clarity, not all edges are drawn. Vertices which are $2$-distance dominated by the first triple are marked with horizontal black strips, vertices which are $2$-distance dominated by the last triple with vertical black strips, and vertices which are $2$-distance dominated by the final vertex with gray color. The remaining white vertices are $2$-distance dominated by the remaining triples.}
        \label{fig:family-2}
    \end{center}
\end{figure}

The next proposition shows that, for every $d$ and $p$ with $p \leq d$, there is an infinite class of strong toruses such that $\gamma_d^p(C_m \boxtimes C_n) < \gamma_d^p(C_m)\, \gamma_d^p(C_m)$.

\begin{theorem} \label{thm:torus-minus-1}
    Let $m$, $n$, $d$, and $p$ be positive integers such that $d \ge p$, $m \ge 2d+2$, and $n \ge 4d+3$. If both $m\! \pmod{2d+1}$ and $n\! \pmod{2d+1}$ are from $[d]$, then
    $$ \gamma_d^p(C_m \boxtimes C_n) \le  \gamma_d^p(C_m)\, \gamma_d^p(C_n)-1.
    $$
\end{theorem}

\proof
Under the conditions in the theorem, we can write $m$ and $n$ in the form
$$m =m_1(2d+1)+q \enskip \mbox{and} \enskip n =n_1(2d+1)+r,
$$
where $m_1= \lfloor\frac{m}{2d+1}\rfloor$, $n_1= \lfloor\frac{n}{2d+1}\rfloor$, and $q, r \in [d]$. 
 
The conditions $m \ge 2d+2$ and $n \ge 4d+3$ imply $m_1\ge 1$ and $n_1 \ge 2$.
By Theorem~\ref{thm:cycles}, $\gamma_d^p(C_m)=m_1+1$ and $\gamma_d^p(C_n)=n_1+1$ for all $p \leq d$. As Proposition~\ref{prop:lower-bound-k-domination} shows $\gamma_d^p(C_m \boxtimes C_n) \leq \gamma_d^d(C_m \boxtimes C_n)$ whenever $p \leq d$, it is enough to prove the statement for $p=d$. So our aim is to construct a $d$-distance $d$-packing set $S$ in $C_m \boxtimes C_n$ that contains $m_1n_1 +m_1 +n_1$ vertices.
\medskip

We define the following vertices and vertex sets, for $i\in [m_1]$ and $j \in [n_1]$:
\begin{itemize}
  \item $z(i,j) =(i(2d+1)-d, j(2d+1)-d),$
  \item $a(i) =(i(2d+1)-2d, n)$ \enskip and \enskip  $b(j) =(m,j(2d+1)),$
  \item $c^* =(m-d, n-r);$
  \item $z^* =(m_1(2d+1)-d, n_1(2d+1)-d-1) = (m-q-d,n-r-d-1),$
  \item $Z =\{ z(i,j) \colon i\in [m_1], j \in [n_1]\},$
  \item $A =\{ a(i) \colon i\in [m_1]\}$, \enskip and \enskip $B =\{ b(j) \colon j\in [n_1-1]\}.$
\end{itemize}
For example, see Fig.~\ref{fig:C16}. It can be readily checked that $A\cup B \cup Z$ is a $d$-packing. The set $Z$ clearly $d$-distance dominates the vertices in $[1, m-q]\times [1, n-r]$. For the remaining vertices, we first observe that the set 
$$([1, m-q-d] \times [n-r+1,n])  \cup ([m-q+1, m] \times [1,n-r-d-1]) 
$$
is $d$-distance dominated by $A\cup B$. In particular, the part $[m-q+1, m] \times [1,d]$ is dominated by  $a(1)$. The vertices that remain undominated by $Z \cup A \cup B$ are all contained in $Y =[m-q-d+1,m] \times [n-r-d, n]$. Since $q \le d$ and $r \leq d$, we have $Y \subseteq [m-2d, m] \times [n-r-d, n]$ and 
$c^*$ can $d$-distance dominate the entire $Y$.
We infer that $S'=Z \cup A \cup B \cup\{c^*\}$ is a $d$-distance dominating set in $C_m\boxtimes C_n$ and its size is $|S'|=m_1n_1+m_1+n_1$.
\medskip

We have already observed that $S'\setminus \{c^*\}$ is a $p$-packing. Further, for every $v \in S'\setminus \{z(m_1,n_1)\}$, the distance $d_{C_m\boxtimes C_n}(c^*,v)$ is at least $d+1$. However,    $$d_{C_m\boxtimes C_n}(c^*,z(m_1,n_1))= (n-r)-(n_1(2d+1)-d) =d
$$ 
that is not large enough for a $d$-packing.
Therefore, we finish the construction by replacing $z(m_1, n_1)$ with the neighbor $z^*$ in the set. Let $S=S' \setminus \{z(m_1,n_1)\} \cup \{z^*\}$. This way $[m]\times [n]$ remains $d$-distance dominated, $d_{C_m\boxtimes C_n}(c^*,z^*)=d+1$, and the distance of $z^*$ and any other vertex in $S$ is at least $d+1$. Observe that the condition $n_1 \ge 2$ ensures $d_{C_m\boxtimes C_n}(a(m_1),z^*) >d+1$. Thus $S$ is a $d$-distance $d$-packing set with $|S|= |S'|=(m_1+1)(n_1+1)-1$, and the statement follows.
\qed

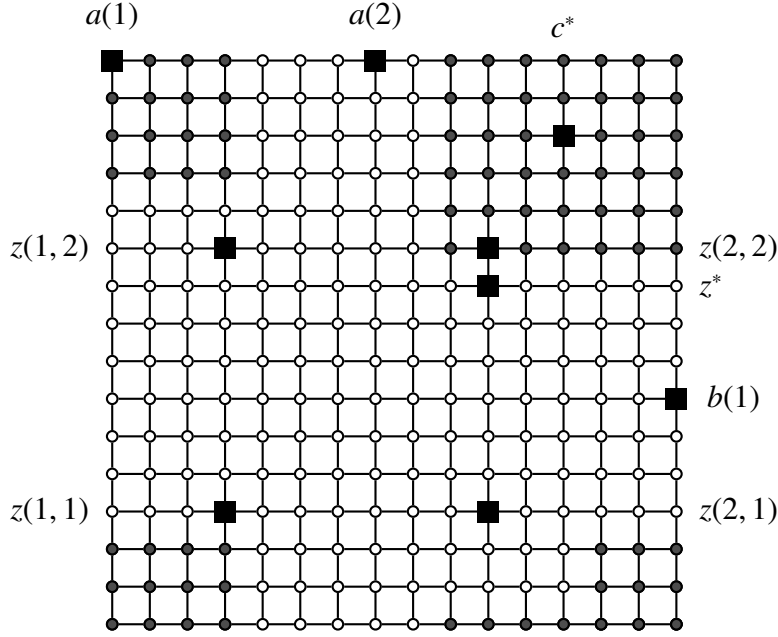
\begin{figure}[!ht]
    \begin{center}
        \begin{tikzpicture}[thick,scale=0.5]
        \tikzstyle{vertex}=[circle, draw, fill=black!10, inner sep=0pt, minimum width=4pt]
        \tikzstyle{vertexBlack}=[rectangle, draw, fill=black, inner sep=0pt, minimum width=7pt]
        \tikzstyle{vertexWhite}=[circle, draw, fill=white, inner sep=0pt, minimum width=4pt]
        \tikzstyle{vertexGray}=[circle, draw, fill=black!70, inner sep=0pt, minimum width=4pt]
        \tikzstyle{vertexBlack2}=[rectangle, rotate=-90, draw, fill=black, inner sep=0pt, minimum width=7pt]
        
        \pgfmathtruncatemacro{\N}{5}
        \pgfmathtruncatemacro{\R}{1}

        \foreach \x in {0,...,15}
            \foreach \y in {0,...,15}
            \node[vertexWhite] (\x+\y) at (\x,\y) {};

        \foreach \y in {0,...,15}
            \foreach \x [remember=\x as \lastx (initially 0)] in {1,...,15}
            \path (\x+\y) edge (\lastx+\y);

        \foreach \x in {0,...,15}
            \foreach \y [remember=\y as \lasty (initially 0)] in {1,...,15}
            \path (\x+\y) edge (\x+\lasty);

        \foreach \x in {9,...,15}
            \foreach \y in {0,10,11,12,13,14,15}
                \node[vertexGray] (\x+\y) at (\x,\y) {};

        \foreach \x in {0,...,3}
            \foreach \y in {12,...,15}
                \node[vertexGray] (\x+\y) at (\x,\y) {};

        \foreach \x in {0,1,2,3,13,14,15}
            \foreach \y in {0,...,2}
                \node[vertexGray] (\x+\y) at (\x,\y) {};

        \node[fill=black, label={[label distance = 1.5cm]left:$z(1,1)$}] (z11) at (3,3) {};
        \node[fill=black, label={[label distance = 2.5cm]right:$z(2,1)$}] (z12) at (10,3) {};
        \node[fill=black, label={[label distance = 1.5cm]left:$z(1,2)$}] (z21) at (3,10) {};
        \node[fill=black, label={[label distance = 2.5cm]right:$z(2,2)$}] (z22) at (10,10) {};
        \node[fill=black, label={[label distance = 2.5cm]right:$z^*$}] (z*) at (10,9) {};
        \node[fill=black, label={[label distance = 0.1cm]above:$a(1)$}] (a1) at (0,15) {};
        \node[fill=black, label={[label distance = 0.1cm]above:$a(2)$}] (a2) at (7,15) {};
        \node[fill=black, label={[label distance = 0.1cm]right:$b(1)$}] (b1) at (15,6) {};
        \node[fill=black, label={[label distance = 1cm]above:$c^*$}] (c*) at (12,13) {};

        \end{tikzpicture}
        \caption{A schematic representation of $C_{16} \boxtimes C_{16}$, i.e.\ $m = n = 16$, $d = p = 3$, $q = r = 2$ and $m_1 = n_1 = 2$. Vertices from $Z \cup A \cup B \cup \{z^*, c^*\}$ are marked with black squares and their labels are written on the side of the graph. For clarity, not all edges are drawn. The vertices $d$-distance dominated by $a(1)$ and by $c^*$ are marked with color gray.}
        \label{fig:C16}
    \end{center}
\end{figure}
\medskip

At the end of this section, we prove a lower bound on $\gamma_d^p(C_m \boxtimes C_n)$ and identify some cases when Theorem~\ref{thm:upper-for-strong} holds with equality for strong toruses that is,
$ \gamma_d^p(C_m \boxtimes C_n) =  \gamma_d^p(C_m)\, \gamma_d^p(C_n)$.

\begin{theorem} \label{thm:eq-for-toruses}
For every two nonnegative integers $d$ and $p$ and for every two cycles $C_m$ and $C_n$, the following statements hold.
\begin{itemize}
    \item[$(i)$] $\left\lceil \frac{m}{2d+1} \right\rceil \frac{n}{2d+1} \le \gamma_d^p(C_m \boxtimes C_n)$.
    \item[$(ii)$] If $n\equiv 0 \pmod{2d+1}$, then $ \gamma_d^p(C_m \boxtimes C_n) = \gamma_d^p(C_m)\, \gamma_d^p(C_n)$.
    \item[$(iii)$] If $n\equiv r \pmod{2d+1}$ and $(1-\frac{r}{2d+1}) \left\lceil \frac{m}{2d+1} \right\rceil<1$, then $ \gamma_d^p(C_m \boxtimes C_n) = \gamma_d^p(C_m)\, \gamma_d^p(C_n)$. In particular, the equality holds if $n\equiv 2d \pmod{2d+1}$ and $m \le 4d^2+2d$, and also if $n\equiv 2d-1 \pmod{2d+1}$ and $m \le 2d^2+d$.
  \end{itemize}    
\end{theorem}

\proof
  By Theorems~\ref{thm:upper-for-strong} and~\ref{thm:strong-torus-infinity}, the statements are valid when $ \gamma_d^p(C_m \boxtimes C_n)=\infty$. Hence, we may assume that $C_m \boxtimes C_n$ has a $\gamma_d^p$-set $S$ and further that $\gamma_d^p(C_m)=\lceil \frac{m}{2d+1} \rceil$ and $\gamma_d^p(C_n)=\lceil \frac{n}{2d+1} \rceil$.
  \medskip
  
  (i) Let $R_i = \{(x,i) \colon x \in [m]\}$ and define $A_i=\bigcup_{j=i-d}^{i+d} R_j$ for every $i \in [n]$, as in the proof of Theorem~\ref{thm:strong-torus-infinity} (see Fig.~\ref{fig:sets-in-product-of-cycles}). Since the vertices in $R_i$ are all $d$-distance dominated by the vertices in $S\cap A_i$, we observe that $|S\cap A_i| \ge \lceil \frac{m}{2d+1}\rceil$ for every $i\in [n]$. Counting the pairs $(i,v)$ complying with $i \in [n]$ and $v \in A_i$, we then infer that this number $a$ is not smaller than $n\lceil \frac{m}{2d+1}\rceil$. On the other hand, every vertex $v$ from $S$ occurs in exactly $2d+1$ such pairs and then $a = |S|\, (2d+1)$.  These observations result in the desired inequality 
  \begin{equation*} \label{eq:cycles-tight}
      \gamma_d^p(C_m \boxtimes C_n) = |S| =\frac{a}{2d+1} \ge  \frac{n}{2d+1}\left\lceil \frac{m}{2d+1}\right\rceil.
  \end{equation*}
 \medskip

(ii) If $n\equiv 0 \pmod{2d+1}$, then $\frac{n}{2d+1}= \lceil \frac{n}{2d+1} \rceil$ and therefore, the lower bound in part (i) and the upper bound in Theorem~\ref{thm:upper-for-strong} coincide. This implies the statement.
\medskip


  (iii) The condition $(1-\frac{r}{2d+1}) \left\lceil \frac{m}{2d+1} \right\rceil<1$ implies $r \neq 0$ and therefore, $\lceil \frac{n}{2d+1}\rceil = \frac{n+2d+1-r}{2d+1}$. 
  Since $\gamma_d^p(C_m \boxtimes C_n)$ and $\lceil \frac{m}{2d+1}\rceil \lceil \frac{n}{2d+1}\rceil$ are integers, they are equal whenever the lower bound  from (i) and the upper bound from Theorem~\ref{thm:upper-for-strong} satisfy
  \begin{align*}
        1 &> \left\lceil \frac{n}{2d+1}\right\rceil \left\lceil \frac{m}{2d+1}\right\rceil -
      \frac{n}{2d+1} \left\lceil \frac{m}{2d+1}\right\rceil\\
      &= \left(\frac{n+2d+1-r}{2d+1}- \frac{n}{2d+1}\right)\left\lceil \frac{m}{2d+1}\right\rceil\\
      &= \left(1-\frac{r}{2d+1}\right) \left\lceil \frac{m}{2d+1} \right\rceil.
  \end{align*}
  As our condition in (iii) corresponds to this inequality, $ \gamma_d^p(C_m \boxtimes C_n) =\lceil \frac{m}{2d+1}\rceil \lceil \frac{n}{2d+1}\rceil$ follows. For $r=2d$, the condition gives $\frac{1}{2d+1} \lceil \frac{m}{2d+1} \rceil <1$ which is equivalent to $\lceil \frac{m}{2d+1} \rceil \leq 2d  $ and hence to $m \leq 4d^2+2d$. A similar reasoning proves the statement for $r=2d-1$.  
\qed

\section{Concluding results and problems}
\label{sec:conclude}

By Theorem~\ref{thm:C11(t+1)}, the difference $\gamma_d^p(G) \gamma_d^p(H) - \gamma_d^p(G\boxtimes H)$ over all connected graphs $G$ can be equal to $2$. We next prove that the difference can be arbitrary. 

\begin{theorem}
\label{thm:arbitrary-difference}    
The difference $\gamma_2^2(G) \gamma_2^2(H) - \gamma_2^2(G\boxtimes H)$ can be an arbitrarily large integer, even if $G$ and $H$ are connected graphs. In particular, this also holds when one factor is required to be a cycle.
\end{theorem}

\begin{proof}
Let $G_k$ be obtained from $k\ge 2$
 vertex disjoint $11$-cycles by gluing them together at two adjacent vertices. Formally, $G_k$ is defined on the vertex set $$V(G_k)=\{v_1, v_2\} \cup \{ v_i^j \colon  i\in [3,11], \enskip j \in [k]\}$$ such that $v_1v_2v_3^j\dots v_{11}^jv_1$ is an $11$-cycle for each $j \in [k]$. We state first that $\gamma_2^2(G_k)=2k$. Indeed, every cycle $C_{11}^j$ contains five consecutive vertices, namely $v_5^j, \dots, v_{9}^j$, which can be $2$-distance dominated only by vertices that belong to $C_{11}^j$ and no other cycles. Supposing that $V(C_{11}^j) \setminus \{v_1,v_2\}$ contains only one vertex from a $2$-distance $2$-packing set of $G_k$, this vertex must be $v_7^j$. However, in this case, $v_4^j$ and $v_{10}^j$ would be $2$-distance dominated by $v_1$ and $v_2$, respectively, and the set would not be a $2$-packing. This contradiction proves $\gamma_2^2(G_k) \ge 2k$. For the other direction, observe that 
 $$ S=\{ v_5^1, v_{10}^1\} \cup 
 \{v_4^j, v_9^j \colon j \in [2,k]\}
 $$ 
is a $2$-distance $2$-packing set in $G_k$ and contains $2k$ vertices. Thus $\gamma_2^2(G_k)=2k$ for every $k \ge 2$. We also know that $\gamma_2^2(C_ {11})=3$.

Consider now the product $F_k= G_k \boxtimes C_{11}$. For each cycle $C_{11} ^j$ in $G_k$, take the strong torus $C_{11}^j \boxtimes C_{11}$ and apply the construction described in the proof of Theorem~\ref{thm:torus-minus-1} to get a $2$-distance $2$-packing set $S^j$ of size $8$ in $C_{11}^j \boxtimes C_{11}$. We can rotate it such that the layers $^{v_1}{C_{11}}$ and $^{v_2}{C_{11}}$ together contain three vertices from $S$. Formally, for every $j \in [k]$, let
$$ S^j=\{(v_1, 3), (v_1,7), (v_2,10), (v_4^j,5), (v_5^j,11), (v_7^j,3), (v_7^j,8), (v_{10}^j,11)\}.
$$
It is clear that $S'= \bigcup_{j=1}^k S^j$ is a $2$-distance dominating set in $F_k$. The distance between any two different vertices $x,y \in S^j$ is $d_{F_k}(x,y) \ge 3$. It can also be checked directly that for every vertex $x \in S^j \setminus \{ (v_1, 3), (v_1,7), (v_2,10)\} $, the  distance between $x$ and the layers $^{v_1}{C_{11}}$ and $^{v_2}{C_{11}}$ is at least two. We conclude that if $x \in S^j \setminus S^{j'}$ and $y \in S^{j'} \setminus S^j$, then $d_{F_k}(x,y) \ge 4$. Therefore, $S'$ is a $2$-packing.
 
Since $S^j \cap S^{j'}=\{(v_1, 3), (v_1,7), (v_2,10)\}$ for every $j \neq j'$, we have $|S'|= 3+5k$. Consequently,
$$\gamma_2^2(G_k \boxtimes C_{11}) \leq 5k+3, 
$$
while $ \gamma_2^2(G_k) \gamma_2^2(C_{11})=3 \cdot 2k = 6k$. Therefore, the difference is not smaller than $k-3$, which proves the statement.
\end{proof}

\begin{remark}  \label{rmk:large-difference-gen}
The statement of Theorem~\ref{thm:arbitrary-difference} can be extended by proving that $\gamma_d^d(G) \gamma_d^d(H) - \gamma_d^d(G\boxtimes H)$ can be arbitrarily large for each $d \ge 2$.  
In this more general proof we assume that $d$ is fixed and define the graph $G_k$, for $k \ge 2$, which is obtained from $k$ cycles $C_{4d+3}^j \colon v_1v_2 v_3^j\dots v_{4d+3}^j v_1$ ($j \in [k])$ such that $V(C_{4d+3}^j) \cap V(C_{4d+3}^{j'})= \{v_1, v_2\} $ whenever $j \neq j'$. One can show, analogously to the proof of Theorem~\ref{thm:arbitrary-difference}, that $\gamma_d^d(G_k)=2k$, for $k \ge 2$. Then, for every $j \in [k]$, we consider the set
\begin{align*} S^j&=\{(v_1, d+1), (v_1,3d+1), (v_2,4d+2), (v_{d+2}^j,2d+1), (v_{d+3}^j,4d+3)\}\\  &\cup \{(v_{2d+3}^j,d+1), (v_{2d+3}^j,3d+2), (v_{3d+4}^j,4d+3)\}.
 \end{align*}
Then $S'= \bigcup_{j=1}^k S^j$ is a $d$-distance $d$-packing set in $G_k \boxtimes C_{4d+3}$, and the proof can be finished with
\begin{align*} \gamma_d^d(G_k \boxtimes C_{4k+3}) &\leq |S'| =3+5k = 3 \cdot 2k - (k-3)
= \gamma_d^d(G_k) \gamma_d^d(C_{4k+3})-(k-3).
 \end{align*}
\end{remark}

In view of Theorem~\ref{thm:arbitrary-difference} and Remark~\ref{rmk:large-difference-gen}, one may ask whether there are graphs such that $\gamma_d^p(G) \gamma_d^p(H)= \infty$ and $\gamma_d^p(G\boxtimes H)$ is finite. Conjecture~\ref{con:infinite} posed in Section~\ref{sec:conjecture} claims the non-existence of such graphs. Here we recall this conjecture, which may be considered the most important open problem related to our topic.

\paragraph{Conjecture~\ref{con:infinite}.} If $G$ is a graph, and $d$ and $p$ are integers such that  $\gamma_d^p(G) = \infty$, then $\gamma_d^p(G\boxtimes H) = \infty$ for every graph $H$.
\medskip

The results of Section~\ref{sec:toruses} suggest our next open problem: 

\begin{problem}
\begin{enumerate}
\item[(i)] For exactly which values of $d$, $p$, $m$, and $n$, the equality $ \gamma_d^p(C_m \boxtimes C_n) = \gamma_d^p(C_m)\, \gamma_d^p(C_n)$ holds? 
\item[(ii)] Is the difference $\gamma_d^p(C_m)\, \gamma_d^p(C_n) - \gamma_d^p(C_m \boxtimes C_n)$ bounded by a constant (possibly $2$) over all values of $d$, $p$, $m$, and $n$?
\end{enumerate}
\end{problem}

In view of Theorem~\ref{thm:equality-grids} we also pose: 

\begin{problem}
Investigate the behavior of $\gamma_d^p(T\boxtimes T')$ where $T$ and $T'$ are arbitrary trees.     
\end{problem}

We conclude the paper with the following unfortunate situation. Since $\gamma_k(G) \le \gamma_k^p(G)$, $p\ge 0$ by Proposition~\ref{prop:lower-bound-k-domination}, we infer that Theorems~\ref{thm:C11(t+1)} and~\ref{thm:torus-minus-1} form counterexamples to~\cite[Theorem 3.3]{anand-2025} which wrongly claims that $\gamma_k(G\boxtimes H) = \gamma_k(G) \gamma_k(H)$. The fundamental error made in the ``proof'' is the claim that if $S_1\subseteq V(G)$, $S_2\subseteq V(H)$, and $S\subseteq V(G\boxtimes H)$ is a set with $|S| < |S_1 \times S_2|$, then $|p_G(S)| < |S_1|$ or $|p_H(S)| < |S_2|$. The paper~\cite{anand-2025} makes the same error for the Cartesian product,~\cite[Theorem 2.5]{anand-2025} gives a short ``proof'' of $\gamma_k(G\cp H) \ge \gamma_k(G)\gamma_k(H)$. If the ``proof'' worked, then  this would mean (by selecting $k=1$) the solution of the celebrated Vizing's conjecture~\cite{bresar-2012, zhao-2022}. Further,~\cite[Theorem 4.3]{anand-2025}  claims that $\gamma_k(G\circ H) = \gamma_k(G)$, where $\circ$ is the lexicographic product. This result is again wrong, for instance, $\gamma(P_6\circ P_4) = 4$, cf.~\cite{zhang-2011}. 

\section*{Data Availability Statement}
 
Data sharing is not applicable to this article as no new data were created or analyzed in this study.

\section*{Conflict of Interest Statement}

Sandi Klav\v{z}ar is a Honorary Board Editor of the European Journal of Combinatorics and was not involved in the review and decision-making process of this article. In addition, the authors declare no other conflict of interest.

\section*{Acknowledgments}

Csilla Bujt\'{a}s, Vesna Ir\v{s}i\v{c}, and Sandi Klav\v{z}ar were supported by the Slovenian Research and Innovation Agency (ARIS) under the grants P1-0297, N1-0285, N1-0355, N1-0431, and  Z1-50003.  Vesna Ir\v{s}i\v{c} also acknowledges the financial support from the European Union (ERC, KARST, 101071836).


\begin{thebibliography}{99}

\bibitem{Abay-2009}
G.~Abay-Asmerom, R.~Hammack, D.T.~Taylor,
Perfect {$r$}-codes in strong products of graphs,
Bull.\ Inst.\ Combin.\ Appl.\ 55 (2009) 66--72.

\bibitem{anand-2025}
B.S.~Anand, J.A.~Dayap, L.F.~Casinillo, R.~Pepper, R.S.~Nair,
On distance $k$-domination number of graphs under product operations,
Gulf J.\ Math.\ 19 (2025) 117--124. 

\bibitem{babu-2024}
P.S.~Babu, A.V.~Chithra,
A note on acyclic coloring of strong product of graphs,
Iran.\ J.\ Math.\ Sci.\ Inform.\ 19 (2024) 149--160.

\bibitem{Beineke-1994}
L.W.~Beineke, M.A.~Henning,
Some extremal results on independent distance domination in graphs,
Ars Combin.\ 37 (1994) 223--233.

\bibitem{bresar-2012}
B.~Bre\v{s}ar, P.~Dorbec, W.~Goddard, B.~Hartnell, M.~A.~Henning, S.~Klav\v{z}ar, D.~F.~Rall,
Vizing's conjecture: a survey and recent results,
J.\ Graph Theory 69 (2012) 46--76.

\bibitem{bresar-2024}
B.~Bre\v{s}ar, J.~Hed\v{z}et, 
Bootstrap percolation in strong products of graphs, 
Electron.\ J.\ Combin.\ 31 (2024) Paper 4.35.

\bibitem{BIKZ-2025a}
Cs.~Bujt\'as, V.~Ir\v si\v c Chenoweth, S.~Klav\v zar, G.~Zhang, 
The $d$-distance $p$-packing domination number: complexity, cycles, and trees,
\url{arXiv:2507.18272v2}  [math.CO] (2025). 

\bibitem{BIKZ-2025b}
Cs.~Bujt\'as, V.~Ir\v si\v c Chenoweth, S.~Klav\v zar, G.~Zhang, 
Revisiting $d$-distance (independent) domination in trees and in bipartite graphs,
\url{arXiv:2508.12804v1} [math.CO] (2025). 

\bibitem{enomoto-2023}
H.~Enomoto, J.~Fujisawa, N.~Matsumoto, 
Game chromatic number of strong product graphs,
Discrete Math.\ 346 (2023) Paper 113162.

\bibitem{esperet-2024}
L.~Esperet, D.R.~Wood, 
Colouring strong products,
European J.\ Combin.\ 121 (2024) Paper 103847p.

\bibitem{Favaron-1992}
O.~Favaron,
A bound on the independent domination number of a tree,
Vishwa Internat.\ J.\ Graph Theory 1 (1992) 19--27.

\bibitem{Gimbel-1996}
J.G.~Gimbel, M.A.~Henning,
Bounds on an independent distance domination parameter,
J. Combin. Math. Combin. Comput.\ 20 (1996) 193--205.

\bibitem{HIK-2011}
R.~Hammack, W.~Imrich, S.~Klav\v{z}ar,
Handbook of Product Graphs, Second Edition,
CRC Press, Boca Raton, FL, 2011.

\bibitem{Haynes-2023}
T.W.~Haynes, S.T.~Hedetniemi, M.A.~Henning,
Domination in Graphs: Core Concepts.
Springer, Cham, 2023.

\bibitem{Henning-1998-1}
M.A.~Henning,
Packing in trees,
Discrete Math.\ 186 (1998) 145--155.

\bibitem{henning-2020}
M.A.~Henning, 
Distance domination in graphs, 
in: Topics in Domination in Graphs (T.W.~Haynes, S.T.~Hedetniemi, M.A.~Henning, eds.),
Springer, Cham (2020) 205--250.

\bibitem{Henning1991}
M.A.~Henning, O.R.~Oellermann, H.C.~Swart,
Bounds on distance domination parameters,
J.\ Combin.\ Comput.\ Inf.\ Sys.\ Sciences 16 (1991) 11--18.

\bibitem{qian-2025}
W.~Qian, F.~Li, 
Exact vertex forwarding index of the strong product of complete graph and cycle, 
Discrete Appl.\ Math.\ 361 (2025) 69--84.

\bibitem{wu-2025}
C.~Wu, Z.~Deng, 
A note on clique immersion of strong product graphs, 
Discrete Math.\ 348 (2025) Paper 114237.

\bibitem{yang-2018}
H.~Yang, X.~Zhang, 
The independent domination numbers of strong product of two cycles,
J.\ Discrete Math.\ Sci.\ Cryptogr.\ 21 (2018) 1495--1507. 

\bibitem{zhang-2011}
X.~Zhang, J.~Liu, J.~Meng, 
Domination in lexicographic product graphs,
Ars Combin.\ 101 (2011) 251--256. 

\bibitem{zhao-2022}
W.~Zhao, R.~Lin, J.~Cai, 
On construction for trees making the equality hold in {V}izing's conjecture,
J.\ Graph Theory 101 (2022) 397--427. 

\end{thebibliography}
\end{document}